\documentclass[smallextended]{article}
\usepackage{amsmath,amssymb,color,url,mathtools}
\usepackage[shortlabels]{enumitem}
\usepackage{amsthm}
\usepackage[mathscr]{euscript}
\usepackage{float}
\usepackage{graphicx}
\usepackage[left=1.5in,right=1.5in,top=1in,bottom=1in]{geometry}
\usepackage{titlesec}

\newtheorem{thm}{Theorem}
\newtheorem{lem}[thm]{Lemma}

\newtheorem{cor}[thm]{Corollary}

\newtheorem{conjecture}{Conjecture}

\theoremstyle{definition}
\newtheorem{definition}{Definition}
\newtheorem{example}{Example}
\newtheorem*{examples}{Examples}

\newtheorem*{remark}{Remark}

\newtheorem*{mainquestion}{Main Question}

\titleformat*{\section}{\large \bfseries}
\titleformat*{\subsection}{\normalsize \bfseries}

\setenumerate[0]{label=(\alph*)}

\newcounter{commentcounter}

\newcommand{\Proj}{\mathbb P}
\newcommand{\R}{\mathbb R}
\newcommand{\RP}{\R \Proj}

\newcommand{\Z}{\mathbb Z}

\newcommand{\rank}{\operatorname{rank}}

\newcommand{\dist}{\operatorname{dist}}
\newcommand{\Gr}{\operatorname{Gr}}
\newcommand{\conv}{\operatorname{Conv}}
\newcommand{\Span}{\operatorname{Span}}

\begin{document}

\title{Fibrations of $\R^3$ by oriented lines}
\author{Michael Harrison\footnote{Portions of this work were completed while the author was in residence at MSRI during the Fall 2018 semester and supported by NSF Grant DMS-1440140.}}

\maketitle

\begin{abstract}
A fibration of $\R^3$ by oriented lines is given by a unit vector field $V : \R^3 \to S^2$, for which all of the integral curves are oriented lines.  A line fibration is called skew if no two fibers are parallel.  Skew fibrations have been the focus of recent study, in part due to their close relationships with great circle fibrations of $S^3$ and with tight contact structures on $\R^3$.  Both geometric and topological classifications of the space of skew fibrations have appeared; these classifications rely on certain rigid geometric properties exhibited by skew fibrations.  Here we study these properties for line fibrations which are not necessarily skew, and we offer some partial answers to the question: in what sense do nonskew fibrations look and behave like skew fibrations?  We develop and utilize a technique, called the parallel plane pushoff, for studying nonskew fibrations.  In addition, we summarize the known relationship between line fibrations and contact structures, and we extend these results to give a complete correspondence.  Finally, we develop a technique for generating nonskew fibrations and offer a number of examples.
\end{abstract}

\section{Introduction and statement of results}

A \emph{fibration of} $\R^3$ \emph{by oriented lines} is given by a unit vector field $V : \R^3 \to S^2$, for which all of the integral curves are oriented lines.

\begin{example} We offer three qualitatively distinct examples.
\label{ex:fibrations}
\begin{enumerate}
\item For any $u \in S^2$, $\R^3$ may be fibered by parallel lines with direction $u$.

\item Choose a linear plane $P_0 \subset \R^3$, foliate $\R^3$ by the parallel planes $P_t$, and foliate each individual plane by parallel lines, whose oriented direction is given by some continuous function $\gamma : \R \to P_0 \cap S^2 = S^1 : t \mapsto \gamma(t)$.

\item Choose an oriented line $\ell$ in $\R^3$, and surround $\ell$ with a family of hyperboloids which foliate $\R^3 - \ell$.  Choosing a ruling of the hyperboloids yields the fibration of $\R^3$ by pairwise skew lines depicted in Figure \ref{fig:hyper}.
\end{enumerate}
\end{example}
\begin{figure}[ht!]
\centerline{
\includegraphics[width=1.55in]{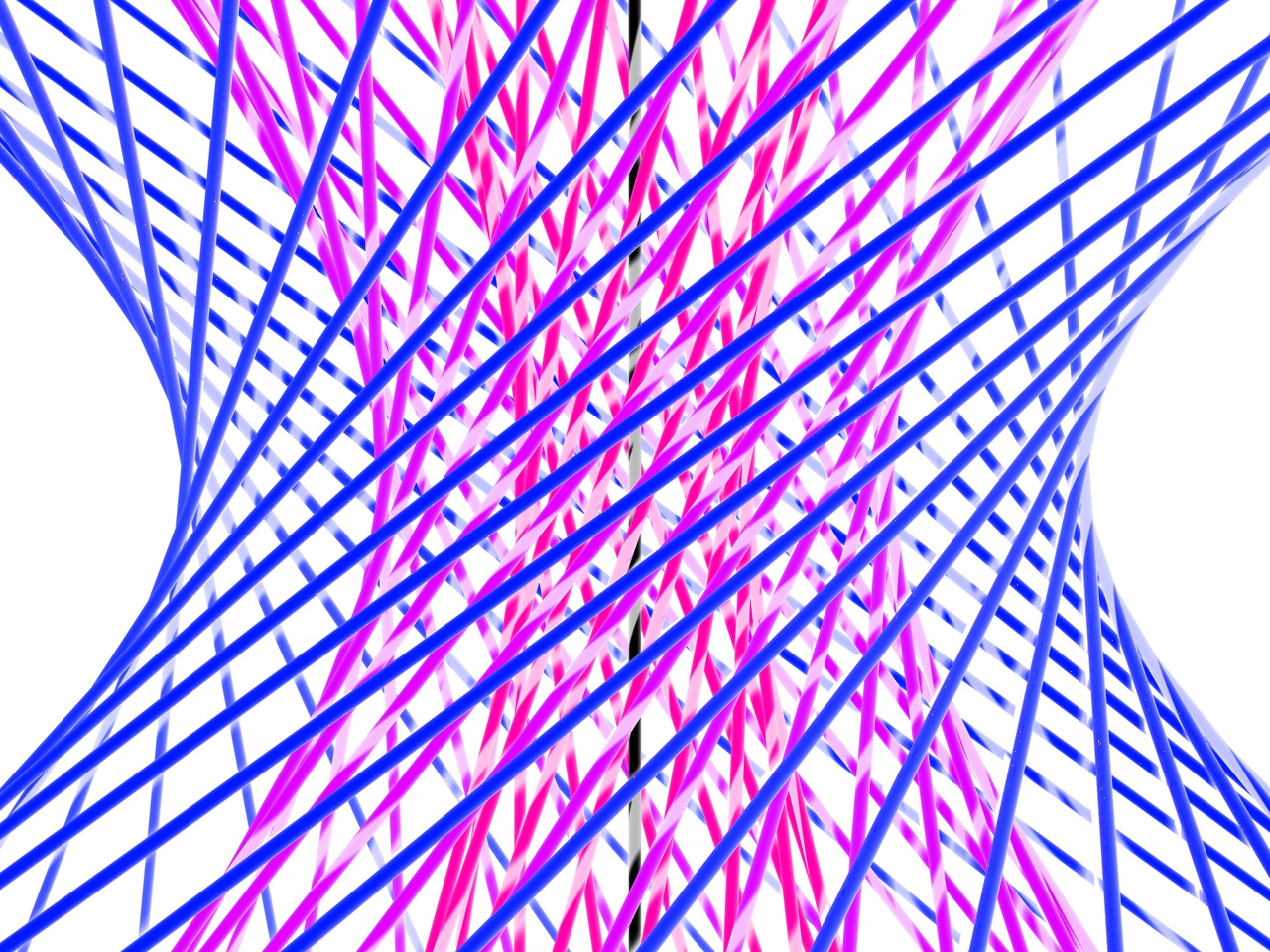}
}
\caption{A skew fibration of $\R^3$ (image by David Eppstein)}
\label{fig:hyper}
\end{figure}

The study of line fibrations originated a decade ago by Salvai \cite{Salvai}, who gave a geometric classification in analogue with the classification of great circle fibrations by Gluck and Warner \cite{GluckWarner}.  More recently, the author studied line fibrations from a topological perspective \cite{Harrison} and studied the relationship of line fibrations with contact structures \cite{Harrison2}; this latter endeavor was continued by Becker and Geiges \cite{BeckerGeiges}.  Ovsienko and Tabachnikov studied fibrations of $\R^n$ by skew affine copies of $\R^p$ (\cite{OvsienkoTabachnikov}, \cite{OvsienkoTabachnikov2}).

Despite the attention received by line fibrations, there are many unanswered questions regarding their rigid geometry.  For example, what subsets of $S^2$ could serve as the image of $V$?  How might such fibrations behave ``at infinity''?  What possible subsets of $\R^3$ could serve as the preimage $V^{-1}(u)$ of any given fiber direction $u$?  Can the image of $V$ contain antipodal points?

These questions have been answered for a special class of line fibrations called \emph{skew fibrations}, whose defining property is that no two fibers are parallel.  The space of skew line fibrations of $\R^3$ is well-studied, in part due to the relationship with the space of great circle fibrations of $S^3$.  For example, the skew fibration depicted in Figure \ref{fig:hyper} is obtained by central projecting the Hopf fibration of $S^3$ to a tangent hyperplane $\R^3$.  For skew fibrations, it is known that the image of $V$ is a convex subset of $S^2$ (\cite{Salvai}) and that skew fibrations exhibit \emph{continuity at infinity} (\cite{Harrison}, see Definition \ref{def:contatinf} in Section \ref{sec:skew}).  By definition of skew fibration, each $V^{-1}(u)$ is a line, and the image of $V$ cannot contain antipodal points.  With these results in mind, the four questions above may be summarized:

\begin{mainquestion}
In what sense do nonskew fibrations look and behave like skew fibrations?
\end{mainquestion}

One class of nonskew fibrations deserves special attention.

\begin{definition}[$1$-parameter fibration]
\label{def:oneparam}
We say that a fibration of $\R^3$ by oriented lines is a \emph{$1$-parameter fibration} if there exists a foliation of $\R^3$ by parallel planes $\left\{P_t\right\}$ which are each individually fibered by lines from the fibration (see Example 1, items (a) and (b)).
\end{definition}

\begin{remark}
The $1$-parameter fibrations for which the direction of the fibers varies at a constant rate are characterized among line fibrations by a property called \emph{fiberwise homogeneity}: for any two fibers $\ell_1$ and $\ell_2$, there is an isometry of Euclidean space which preserves the fibration and takes $\ell_1$ to $\ell_2$ (see \cite{Nuchi}).  The corresponding property on spheres characterizes the Hopf fibrations \cite{Nuchi2}.  
\end{remark}

Our main result is a partial answer to the main question above, and it provides a structural classification for a large class of nonskew fibrations.

For $u \in U \coloneqq V(\R^3)$, we use $u^\perp$ to represent the plane orthogonal to $u$ and passing through the origin of $\R^3$.  We define the closed subset $S_u \subset u^\perp$ as the intersection $u^\perp \cap V^{-1}(u)$.  We occasionally refer to $S_u$ as a \emph{base space}, since $V^{-1}(u)$ is geometrically the direct product $S_u \times \R$.

\begin{thm}
\label{thm:main}
Consider a fibration of $\R^3$ by oriented lines given by a unit vector field $V: \R^3 \to S^2$, and consider $S_u$ defined above.

\begin{enumerate}
\item If the fibration is not $1$-parameter, then every $S_u$ is convex (in particular, connected).

\item If $S_u$ is compact, then \emph{continuity at infinity} (see Definition \ref{def:contatinf}) holds for the direction $u$, and $-u \notin U$.

\item If for all $u \in U$, $S_u$ is compact, then $U$ is a convex subset of $S^2$.

\item If for all $u \in U$, $S_u$ is not compact, then every $S_u$ is a union of one or more lines, hence the fibration is $1$-parameter.
\end{enumerate}
\end{thm}

The theorem might be interpreted as follows.  Item (c) describes the situation in which the fibration is skewlike; i.e.\ it looks and behaves as a skew fibration, except for possibly countably many nonskew patches.  Item (d) describes the $1$-parameter fibrations.  The theorem makes no claim about the behavior of fibrations for which there coexist some compact base spaces with some noncompact base spaces.  These seem to be the most mysterious among the line fibrations, and we can offer only partial results.

For example, if there exists precisely one $u \in U$ for which $S_u$ is noncompact, then the techniques developed here imply that the fibration is skewlike outside of $S_u$, and the shape of the convex set $S_u$ dictates certain qualitative behaviors of the fibration and the subset $U$.  On the other hand, we may construct an example with an open great semicircle worth of directions with noncompact base, by gluing half of a $1$-parameter fibration to half of a skew fibration; see Example \ref{ex:halfhalf}.  Finally, there exists a class of \emph{exotic} fibrations, for which $U$ contains a pair of antipodal points $\pm u$, and the sets $S_u$ and $S_{-u}$ determine uniquely the closure of $U$ (see Section \ref{sec:exotic}).  These exotic fibrations are the only known non-$1$-parameter fibrations which contain an antipodal pair $\pm u \in U$.

In light of these examples, we offer two conjectures.

\begin{conjecture}
\label{conj:one}
If there exist multiple $u \in U$ for which $S_u$ is noncompact, then the set of such $u$ is contained in some great circle in $S^2$.
\end{conjecture}

\begin{conjecture}
\label{conj:two}
If a fibration is not $1$-parameter, there is at most one pair of antipodal points in $U$.  If such a pair $\pm u$ exists, then $S_v$ is compact for every $v \in U - \left\{ \pm u \right\}$, and $U - \left\{ \pm u \right\}$ is convex.
\end{conjecture}

If true, the conjectures allow us to completely understand nonskew fibrations in four categories: skewlike, $1$-parameter, half-and-half, and exotic.  In Section \ref{sec:additional} we discuss partial results related to the conjectures.

One consequence of Theorem \ref{thm:main} is a simple classification of analytic fibrations.

\begin{cor}
\label{cor:analytic}
If a fibration of $\R^3$ is given by an analytic vector field $V$, then it is either a $1$-parameter fibration or a skew fibration.
\end{cor}

\begin{proof} If the fibration is not skew, then for some $u \in U$, $S_u$ is not a single point.  By convexity $S_u$ contains a line segment.  Then by analyticity, $S_u$ contains a line, and the fibration must be $1$-parameter.
\end{proof}

A second goal of this paper is to completely understand the relationship between line fibrations and contact structures.  A contact structure on $\R^3$ is a maximally non-integrable plane field $\xi$.  Any contact structure may be defined as the kernel of a $1$-form $\alpha$ with $\alpha \wedge d\alpha$ never zero.  A contact structure is called \emph{overtwisted} if there exists an embedded disk $D$ in $\R^3$ such that $\partial D$ is tangent to $\xi$ while $D$ is transverse to $\xi$ along $\partial D$.  Otherwise $\xi$ is called \emph{tight}.  A contact structure $\xi$ on $\R^3$ is \emph{tight at infinity} if it is tight outside a compact set.  The contact structures on $\R^3$ were classified by Eliashberg in \cite{Eliashberg}: up to isotopy, there is one tight contact structure, one overtwisted contact structure which is not tight at infinity, and a countable number of pairwise non-isotopic overtwisted contact structures which are tight at infinity.

To each line fibration $V$ there corresponds a plane distribution $\xi_V$, defined by $\xi_V(p) = (V(p))^\perp$.

\begin{definition} A smooth line fibration given by a unit vector field $V : \R^3 \to S^2$ is called \emph{nondegenerate} if $\nabla V$ vanishes only in the direction of $V$; equivalently, if $\nabla V |_{V^\perp}$ has rank $2$ everywhere.  A fibration is called \emph{semidegenerate} if $\nabla V |_{V^\perp}$ has rank at least $1$ everywhere.
\end{definition}

Nondegeneracy is a form of local skewness, though in fact any nondegenerate fibration is globally skew.  This fact was originally proven by Salvai, who classified the nondegenerate fibrations in \cite{Salvai}; though a succinct and more direct proof was recently given by Becker and Geiges \cite{BeckerGeiges}.

The author showed in \cite{Harrison2} that the plane distribution induced by a nondegenerate fibration is a tight contact structure.  At first glance, the contact assertion appears to be a purely local statement, in the sense that it claims that the first-order nondegeneracy condition implies the first-order contact condition.  However, consider the vector field $V(p) = \frac{p}{|p|}$ defined on $\R^3 - \left\{ 0 \right\}$.  Then $V$ induces a foliation of $\R^3 - \left\{ 0 \right\}$ by oriented rays, and $V$ is nondegenerate at every point of its domain, but the induced plane field is not a contact distribution -- round spheres are tangent.  Thus the globality of a line fibration gives some additional structure which contributes to the local contact condition.

The tightness argument uses a result of Etnyre, Komendarczyk, and Massot \cite{EtnyreEtal} on tightness in contact metric $3$-manifolds.  In fact, it is shown that if a contact structure is induced from a line fibration for which at least one fiber admits no parallel fibers, then the contact structure is tight.  More recently, Becker and Geiges showed that if a line fibration has the property that every fiber admits a parallel fiber, then $\nabla V$ has rank $1$ at every point $p$, and the line fibration corresponds to a tight contact structure.  Thus from \cite{Harrison2} and \cite{BeckerGeiges}, it is known that:
\begin{itemize}
\item Any contact structure induced by a line fibration is tight.
\item If a line fibration is nondegenerate, i.e.\ $\rank(\nabla V(p)) = 2$ for all $p$, the induced plane distribution is contact.
\item If a line fibration satisfies $\rank(\nabla V(p)) = 1$ for all $p$, the induced plane distribution is contact.
\end{itemize}

We complete the relationship between line fibrations and contact structures with the theorem below.  We also offer an alternate proof of the tightness result of Becker and Geiges based on the structural classification of Theorem \ref{thm:main}.

\begin{thm}
\label{thm:contact}
The plane distribution associated to a smooth line fibration of $\R^3$ is a contact structure if and only if the fibration is semidegenerate.  Moreover, any such contact structure is tight.
\end{thm}

The proof uses a certain local description of a line fibration of $\R^3$.  Let $p \in \R^3$ be contained in a fiber $\ell$ with direction $u$.  There exists a neighborhood $E$ of $p$ in (any copy of) $u^\perp$, and a map $B : E \to \R^2$ which generates the fibration near $\ell$, in the sense that the fiber through $q \in E$ is the graph of the affine map $t \mapsto q + tB(q)$.  This description is a priori local, since fibers away from $\ell$ may lie in $u^\perp$.

\begin{figure}[h!t]
\centerline{
\includegraphics[width=3in]{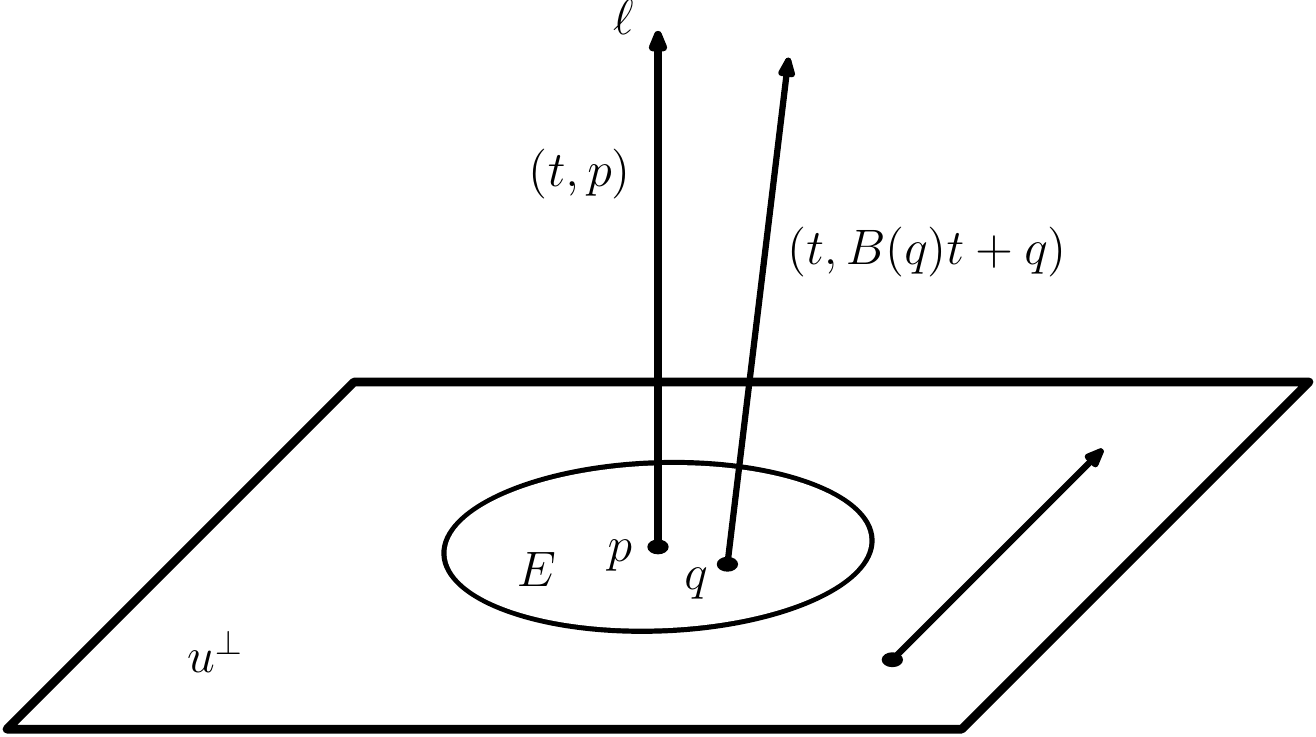}
}
\caption{Local description of a line fibration}
\label{fig:local}
\end{figure}

\begin{thm}
\label{thm:nononzero}
Consider a smooth fibration of $\R^3$ by oriented lines, let $p \in \R^3$ be a point with fiber direction $u$, and consider the map $B$ defined in a neighborhood $E$ of $p$ in $u^\perp$.  Then $dB_p$ has no nonzero real eigenvalues.  In particular:
\begin{enumerate}
\item (Harrison \cite{Harrison2}) If the fibration is nondegenerate at $p$, then the eigenvalues of $dB_p$ are complex.
\item Otherwise, both eigenvalues of $dB_p$ are zero.
\end{enumerate}
\end{thm}

\begin{remark}
\label{rem:skewprop}
One may formulate the corresponding statements for continuous fibrations.  The analogous statement for the first item is that the fibers through $E$ are pairwise skew if and only if, for all $q_1 \neq q_2$ in $E$, $\det\big( q_1 - q_2 \ \ B(q_1)-B(q_2) \big) \neq 0$; see \cite{Harrison}.  If not necessarily skew, there still holds the non-intersecting line property: for all $q_1 \neq q_2$ in $E$, $q_1 - q_2 \neq \lambda(B(q_1) - B(q_2))$ for any $\lambda \in \R$.
\end{remark}

We will prove Theorem \ref{thm:nononzero} in Section \ref{sec:contact}; its primary use will be to verify the contact assertion for semidegenerate fibrations.

We conclude this section with a topological classification of the space of line fibrations.  In \cite{Harrison}, the first author showed that the space of skew fibrations deformation retracts to its subspace of Hopf fibrations, and therefore has the homotopy type of $S^2 \sqcup S^2$.  The classification of all line fibrations is much simpler.

\begin{thm} The space of all line fibrations of $\R^3$ deformation retracts to the space of trivial fibrations $\R^2 \times \R$, and thus has the homotopy type of $S^2$.
\end{thm}

\begin{proof}  Choose an origin $0$ of $\R^3$.  Then any fibration given by a unit vector field $V(x,y,z)$ can be homotoped to the constant fibration with direction $V(0)$ via the map $V_t(x,y,z) = V((1-t)x,(1-t)y,(1-t)z)$.  The endpoint varies continuously with $V$ because the direction of the fiber through the origin changes continuously with respect to the fibration.
\end{proof}

In Section \ref{sec:topprop} we study the structural and topological properties of line fibrations as required for Theorem \ref{thm:main}.  The highlight is the technical investigation of the parallel plane pushoff map, an important and powerful tool used in the proof of Theorem \ref{thm:main}.  In Section \ref{sec:contact} we focus on semidegenerate fibrations and contact structures, and there we prove Theorems \ref{thm:contact} and \ref{thm:nononzero}.  Sections \ref{sec:topprop} and \ref{sec:contact} are independent, besides the topological properties of line fibrations which are required for the tightness assertion.  Finally, in Section \ref{sec:examples}, we provide a method for constructing nontrivial line fibrations, and we offer a number of examples, including the exotic fibrations discussed following Theorem \ref{thm:main}. \\

\noindent \textbf{Acknowledgments.}  We are grateful to Albert Fathi, Emmy Murphy, and Sergei Tabachnikov for a number of stimulating discussions on line fibrations.

\section{Topological features of line fibrations}
\label{sec:topprop}

\subsection{Structural results}

We begin with two simple but important structural results which will help our classification.  As in the previous section, we consider a fibration defined by a unit vector field $V$.  For $u \in U \coloneqq V(\R^3) \subset S^2$, $u^\perp$ refers to the linear plane orthogonal to $u$, and $S_u = V^{-1}(u) \cap u^\perp$ is the points of $u^\perp$ with fiber direction $u$.

\begin{lem}
\label{lem:convex1}
Let $u \in U$ and let $S$ be a connected component of  $S_u$.  Then the convex hull of $S$ is also contained in $S_u$, and hence each connected component of $S_u$ is convex.
\end{lem}

\begin{proof}
Let $p \in \operatorname{Conv}(S) - S$.  There exist two points $p_1$, $p_2$ of $S$ such that the line containing $p_1$ and $p_2$ contains $p$.  Every other line in $u^\perp$ through $p$ must intersect $S$, since each such line separates $p_1$ from $p_2$.  Hence the fiber through $p$ will intersect $S \times \R$ unless it also has direction $\pm u$.  Since $\operatorname{Conv}(S)$ is connected, $V|_{\operatorname{Conv}(S)}$ is constant and equal to $u$.
\end{proof}

\begin{lem}
\label{lem:discbyline}
If $S_1$ and $S_2$ are any two distinct connected components of $S_u$, then there exists a line separating $S_1$ and $S_2$ which does not intersect $S_u$. 
\end{lem}

\begin{proof}
We prove the contrapositive.  If there exists no such line, then every line which intersects $\operatorname{Conv}(S_1 \cup S_2)$ also intersects $S_u$.  By the same argument of Lemma \ref{lem:convex1}, $\operatorname{Conv}(S_1 \cup S_2) \subset S_u$, so $S_1$ and $S_2$ are not distinct connected components.
\end{proof}

\subsection{Continuity properties for skew fibrations}
\label{sec:skew}

Here we recall several features of skew fibrations.  In Section \ref{sec:nonskew} we will investigate in which sense these results hold for nonskew fibrations.

\begin{definition}[Continuity at Infinity]
\label{def:contatinf}  Given a line fibration of $\R^3$, we say that a direction $u \in U$ exhibits \emph{continuity at infinity} if, for every sequence $\left\{p_n\right\} \subset \R^3$ such that $|p_n| \to \infty$ and $\frac{p_n}{|p_n|} \to \pm u$, we have $V(p_n) \to u$.
\end{definition}

Geometrically, the hypotheses $|p_n| \to \infty$ and $\frac{p_n}{|p_n|} \to \pm u$ state that all but finitely many $p_n$ lie in each truncated double-cone
\[\left\{p \in \R^3 \ \Big| \ |p| \geq N, \ \dist_{\RP^2}\left(\frac{p}{|p|}, \pm u\right) \leq \delta\right\},
\]
One may imagine that, for a skew fibration, some small foliated neighborhood of $\ell \coloneqq V^{-1}(u)$ eventually consumes this double-cone, half of which is depicted in Figure \ref{fig:contatinf}.  Hence everything in the cone has fiber direction near $u$, and skew fibrations exhibit continuity at infinity for every direction $u \in U$.  In the opposite situation, the $1$-parameter fibrations do not exhibit continuity at infinity for any direction $u \in U$, except in the case of the trivial fibration $\R^2 \times \R$.  If a line fibration exhibits continuity at infinity with respect to $u$, then it is easy to see that $-u \notin U$; for otherwise, if $\ell'$ is a fiber with direction $-u$, then any diverging sequence of points on $\ell'$ contradicts continuity at infinity for $u$.

\begin{figure}[h!t]
\centerline{
\includegraphics[width=3in]{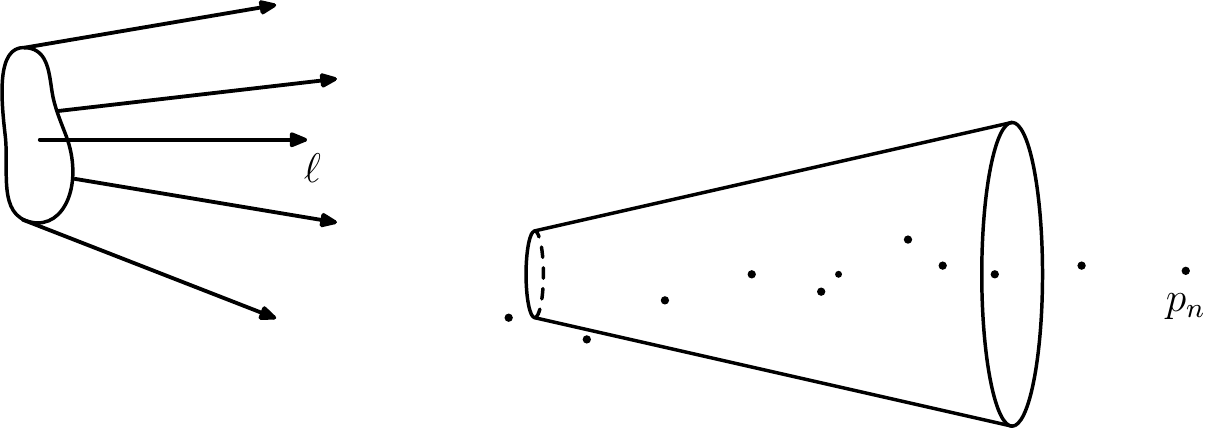}
}
\caption{A foliated neighborhood of $\ell$ consumes everything in the $u$ direction}
\label{fig:contatinf}
\end{figure}

\begin{lem}[Harrison \cite{Harrison}]
\label{lem:continfskew}
Given a skew line fibration of $\R^3$, every $u \in U$ exhibits continuity at infinity.
\end{lem}

We do not repeat the proof here, but it follows the intuitive idea offered in Figure \ref{fig:contatinf}.  It is worth keeping in mind the following concrete example.

\begin{example} Suppose that $\R^3 \to U$ is the Hopf fibration depicted in Figure \ref{fig:hyper}, so that $U$ is the upper open hemisphere of $S^2$.  Let $u \in U$ be the north pole, which corresponds to the direction of the line through the origin of $\R^3$.  Let $(x,y,z)$ be rectangular coordinates on $\R^3$ and fix a pair $(x_0,y_0)$.  Then the sequence of directions $u_n$ corresponding to the points $p_n = (x_0,y_0,n)$ must converge to $u$ as $n \to \pm\infty$, independent of the values of $x_0$ and $y_0$.
\end{example}


Next we introduce a notion of continuity which is weaker than continuity at infinity but similarly important.  While continuity at infinity is stated in reference to a direction $u \in U$, this continuity is stated in reference to a given $2$-plane $P \in \Gr_2(3)$.

Let $P_0$ be the copy of $P$ which passes through the origin of $\R^3$, and let $P_t$ be the parallel translate of $P$ which is distance $t$ from the origin. Observe that there are three possible ways in which a plane $P_t$ can interact with the fibration: all fibers are transverse to $P_t$, or $P_t$ contains a single fiber, or $P_t$ contains multiple fibers which are necessarily parallel.  Let $A_t \subset P_t$ consist of points whose fibers are contained in $P_t$.  In the first case, $A_t$ is empty; in the second, a line; and in the third, a union of lines.

For each point $p \in P_t - A_t$, the unique fiber through $p$ is transverse to the plane $P_t$.  In particular, for any values of $s$ and $t$, the sets $P_t - A_t$ and $P_s - A_s$ are homeomorphic, by the map which follows transverse fibers from their intersection with $P_t - A_t$ to their intersection with $P_s - A_s$, or vice versa.  In particular, for each $t$ and $s$, there is a one-to-one correspondence between the connected components of $P_t - A_t$ and those of $P_s - A_s$.

In the skew case, each $P_t$ can contain at most one fiber.  If any plane $P_t$ contains a fiber, then so must every $P_s$, and the function $\gamma_P : \R \to P_0 \cap S^2$, which sends the parameter $t$ to the direction of the fiber contained in $P_t$, is well-defined.

\begin{lem}
Given a skew line fibration of $\R^3$, let $P \in \Gr_2(3)$ be any $2$-plane which is not transverse to all fibers.  Then $\gamma_P : \R \to P_0 \cap S^2$, which sends the parameter $t$ to the direction of the fiber contained in $P_t$, is continuous.
\end{lem}

\begin{proof} Without loss of generality we show continuity at $t=0$.  Let $u$ be the oriented direction of the fiber $\ell$ contained in $P_0$, and let $p = \ell \cap u^\perp$.  Then the map $V$, restricted to a neighborhood $E \subset u^\perp$ of $p$, is continuous and injective, so by Invariance of Domain, it is a homeomorphism onto its image $V(E) \subset U$.  The great circle $P_0 \cap S^2$ intersects $V(E)$ in a small great circular arc $\alpha$ through $u$, and $(V\big|_E)^{-1}(\alpha)$ is a curve $\beta \subset u^\perp$ containing $p$.  Moreover, for sufficiently small $t$, the curve $\beta$ intersects each plane $P_t$ exactly once, since each $P_t$ contains exactly one fiber, with direction necessarily from $\alpha$.  In this way we may consider $\beta$ parametrized by $t$.  Then $\gamma_P(t) = V(\beta(t))$ is continuous.
\end{proof}

\begin{figure}[h!t]
\centerline{
\includegraphics[width=3.5in]{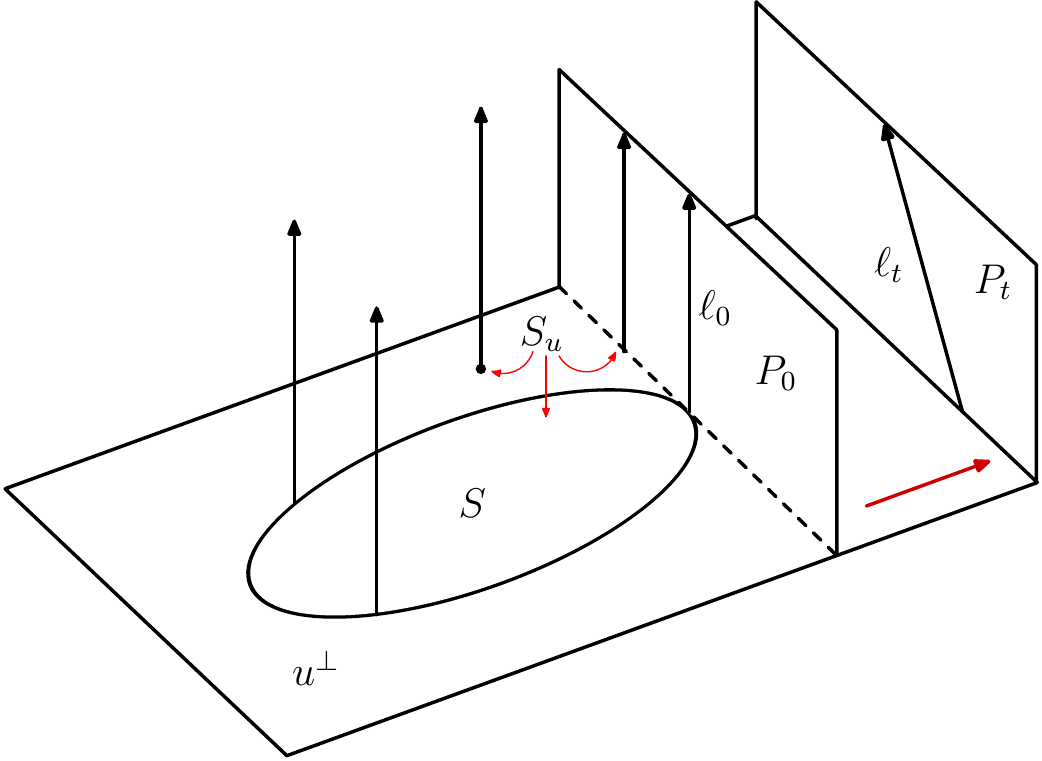}
}
\caption{Pushing off of $S$ by the plane $P$.  Although here we have depicted $S_u$ as disconnected, we will show later that this cannot be the case.}
\label{fig:ppp}
\end{figure}

\begin{lem}
\label{lem:convex}
Given a skew line fibration of $\R^3$, $U$ is convex.
\end{lem}

\begin{proof}
Given distinct $u_1, u_2 \in U$, let $P$ be the plane spanned by the vectors $u_1$ and $u_2$.  Then $\gamma_P$ is continuous, so its image is a connected arc of the great circle $S^2 \cap P$ with no antipodal points; hence the shorter segment of the great circle connecting $u_1$ and $u_2$ is inside $U$.
\end{proof}

\subsection{Continuity of the parallel plane pushoff}
\label{sec:nonskew}

To motivate the study of $\gamma_P$ for nonskew fibrations, consider the following: if $u \in U$ is some direction in a nonskew fibration, and if $p \in S_u$ is a boundary point of some (necessarily convex, by Lemma \ref{lem:convex1}) connected component $S$ of $S_u$, then we can choose a supporting line $m \ni p$ of $S$ and consider the plane $P_0$ spanned by $m$ and $u$, which contains the fiber through $p$.  Now we would like to ``push off" of $p$ in the direction perpendicular to $P_0$ and would like to make some statements about the directions of fibers contained in $P_t$.  We often refer to this process, depicted in Figure \ref{fig:ppp}, as the Parallel Plane Pushoff.

\begin{definition}
A plane $P \in \Gr_2(3)$ is \emph{disconnected by fibers} if some (and hence every) translate $P_t$ is disconnected by one or more fibers contained in $P_t$.
\end{definition}

For example, a plane containing a single fiber is disconnected by fibers.  A plane which contains a half-plane of parallel fibers and no other fibers is \emph{not} disconnected by fibers.

Consider a plane $P \in \Gr_2(3)$ disconnected by fibers.  Since no plane can contain two nonparallel fibers, each translate $P_t$ contains only one (unoriented) direction of fibers.  Let $\tilde{\gamma}_P : \R \to (P \cap S^2) / \Z_2$ send $t$ to the unoriented direction of the fibers contained in $P_t$.

\begin{lem}[Parallel Plane Pushoff]
\label{lem:ppp}
Given a line fibration of $\R^3$ and any plane $P \in \Gr_2(3)$ disconnected by fibers, the parallel plane pushoff map $\tilde{\gamma}_P : \R \to (P \cap S^2) / \Z_2$ is continuous.
\end{lem}

Recall that $A_t \subset P_t$ is defined as the set of points whose containing fibers are contained in $P_t$.  We occasionally abuse language by referring to a sequence when it may be necessary to pass to a convergent subsequence.

We refer the reader to Figure \ref{fig:ppp} for a depiction of this lemma, and we offer some intuition for the proof.  We argue continuity at $t = 0$.  If there is a convergent sequence of points $p_n \in A_{t_n}$ for some sequence $t_n \to 0$, then the statement follows from continuity of the fibration.  Otherwise the sets $A_{t_n}$ are ``diverging'' as $t_n \to 0$.  But there is a homeomorphism $P_0 - A_0 \to P_{t_n} - A_{t_n}$ which respects connected components, and divergence of $A_{t_n}$ would imply that the homeomorphism sends points infinitely far away in small finite time.  This is depicted in Figure \ref{fig:pppproof}; $q_1$ and $q_2$ are in distinct connected components of $P_0 - A_0$, and the fibers through these points correspond to distinct connected components of $P_{t_n} - A_{t_n}$.

\begin{proof} We prove continuity of $\tilde{\gamma}_P$ at $t=0$.  By the hypothesis that $P_0$ is disconnected by fibers, $P_0$ contains a fiber with some direction $u$, and $P_0 \cap S_u$ contains a bounded connected component $Y$.  Choose a small $\varepsilon$ neighborhood $I$ of $Y$ inside the line $P_0 \cap u^\perp$ containing no points from $V^{-1}(-u)$, if any exist.  Note that $I$ may possibly contain points from other connected components of $P_0 \cap S_u$.  In Figure \ref{fig:pppproof}, $I$ is the dotted line containing $Y$.

\begin{figure}[h!t]
\centerline{
\includegraphics[width=3.5in]{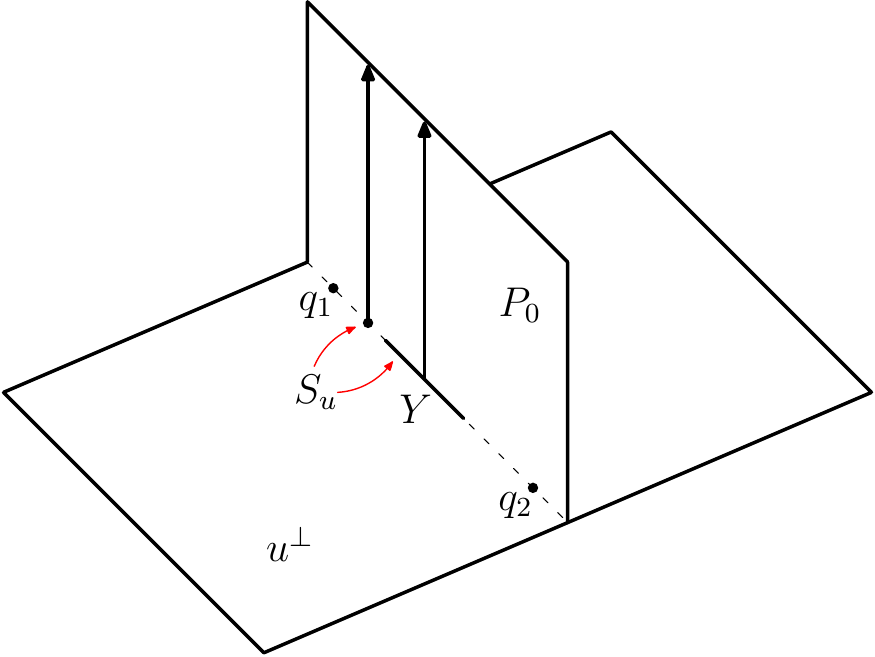}
}
\caption{The fibers emanating from $q_1$ and $q_2$ correspond to distinct connected components of $P_t - A_t$ and cannot go infinitely far in finite time.  Although here we have depicted $S_u$ as disconnected, we will show later that this cannot be the case.}
\label{fig:pppproof}
\end{figure}

Now $Y$ separates $P_0 \cap u^\perp$, so $I - V^{-1}(u)$ is disconnected.  Let $q_1$ and $q_2$ be points in distinct connected components of $I - V^{-1}(u)$.  Let $t_n \to 0$ and observe that each of $q_1$ and $q_2$ correspond to a connected component $Q_{n_1}$ and $Q_{n_2}$ of $P_{t_n} - A_{t_n}$.  Choose any fiber $\ell_n$ in $P_{t_n}$ which separates $Q_{n_1}$ from $Q_{n_2}$, in the sense that $Q_{n_1}$ and $Q_{n_2}$ lie in distinct connected components of $P_{t_n} - \ell_n$.  Let $u_n$ be the oriented direction of $\ell_n$ and $p_n$ be the point of $\ell_n$ at which the minimum distance from $\ell_n$ to $Y$ is achieved.

If the sequence $p_n$ is bounded, then there exists a convergent subsequence to a point $p' \in P_0$.  Therefore by continuity of the fibration, the sequence $u_n$ converges to the direction $u'$ of the fiber through $p'$.  Now $u_n$ is a sequence of points in the great circle $P \cap S^2$, so the limit $u'$ must also lie in this circle, and so the fiber through $p'$ must be contained in $P_0$, hence $u'$ must be either $\pm u$.

Let $\dist(p',Y) = C$.  If $C = 0$, then $p'$ is in $Y$, and $u' = u$.  If $C < \varepsilon$, then the fiber through $p'$ must intersect $I$, in which case it must have direction $u$, since $I$ was chosen to contain no points from $V^{-1}(-u)$.

There are two remaining cases: either $C \geq \varepsilon$, or the sequence $p_n$ is unbounded, but this latter case also implies the existence of $C \geq \varepsilon$ such that $\dist(p_n,Y) > C$ for all $n$ sufficiently large.  Therefore, in either case, the sequence of distances from $Y$ to one of the connected components, say $Q_{n_1}$, is bounded below by $C$.  But we have $\dist(q_1, Y) < \varepsilon \leq C$, so we can choose a sufficiently small time $t'$, such that the endpoint of the fiber segment, emanating from $q_1$ and continuing for time $t'$, is still less than distance $C$ from $Y$.  So it is impossible that this fiber segment touches $Q_{n_1}$ for times $t_n$ sufficiently small.
\end{proof}

In fact, we have not only proven the continuity of the unoriented map $\tilde{\gamma}$, but we have shown that for any $Y$ as above, and for any $\varepsilon > 0$, there is a sequence of points $p_n \in A_{t_n}$, converging to a point $p$ with direction $u$, such that $p$ is within distance $\varepsilon$ of $Y$.  Repeating this with smaller values of $\varepsilon$ and using a diagonal argument establishes that there exists a sequence of points in $A_{t_n}$ which converge to a point in $Y$. This idea will allow us to define an \emph{oriented}, continuous map $\gamma_P$ as follows.

Let $u \in U$, $P$ a plane containing $u$ which is disconnected by fibers, $u^\perp$ the linear plane orthogonal to $u$, $S$ some connected component of $S_u$, and $A$ the subset of $u^\perp$ consisting of points which are contained in fibers fully contained in $u^\perp$, so that $A$ is a (possibly empty) union of lines.  Let $E$ be the connected component of $u^\perp - A$ containing $S$.  When $A$ is empty, $E = u^\perp$; otherwise, $E$ could be a half-plane or strip between two parallel lines.  In this case, the image $V(E)$ must be contained in the open hemisphere of $S^2$ which is centered at $u$, since $V(E)$ cannot intersect the equator orthogonal to $u$.

Let $P_0$ be any translate of $P$ which intersects $S$.  Define the map $\gamma_{P,E} : \R \to P \cap S^2$ which sends the parameter $t$ to the unique direction of the line(s) contained in $P_t$ and intersecting the interval $P_t \cap E$.  Note that $\gamma_{P,E}$ is well-defined: there is at least one candidate output, since the map $P_0 - A_0 \to P_t - A_t$ respects connected components, and there are not two possible outputs, because the image $V(E)$ is contained in a hemisphere.


Now $\gamma_{P,E}$ is just the oriented counterpart to $\tilde{\gamma}_P$ of Lemma \ref{lem:ppp}, and since the restriction to $E$ is contained in a hemisphere, continuity of $\gamma_{P,E}$ follows from Lemma \ref{lem:ppp}.  We have shown the following.

\begin{cor}
\label{cor:localppp}
Given $u \in U$, let $S \subset u^\perp$ be any connected component of $S_u$ and $E$ as defined above.  Then for every $P \in \Gr_2(3)$ that contains $u$ and is disconnected by fibers, the \emph{oriented} map $\gamma_{P,E} : \R \to S^2 \cap P$ is defined and continuous.
\end{cor}

\begin{lem}[No Canyon Lemma]
\label{lem:nocanyon}
Consider a line fibration which is not a $1$-parameter fibration.  Let $u \in U$, $S \subset u^\perp$ be any connected component of $S_u$ and $E$ as defined above.  Then for every $P \in \Gr_2(3)$ that contains $u$ and is disconnected by fibers, $\gamma_{P,E}$ is monotone, considered as a map into the open great semicircular subset of $P \cap S^2$ centered at $u$. 
\end{lem}

Geometrically, the lemma prohibits canyons as depicted in Figure \ref{fig:canyon}.  Canyons are prohibited because points which are trapped "inside" the canyon can have no direction except parallel to the canyon, which cannot happen unless the entire fibration is $1$-parameter.

\begin{figure}[h!t]
\centerline{
\includegraphics[width=2in]{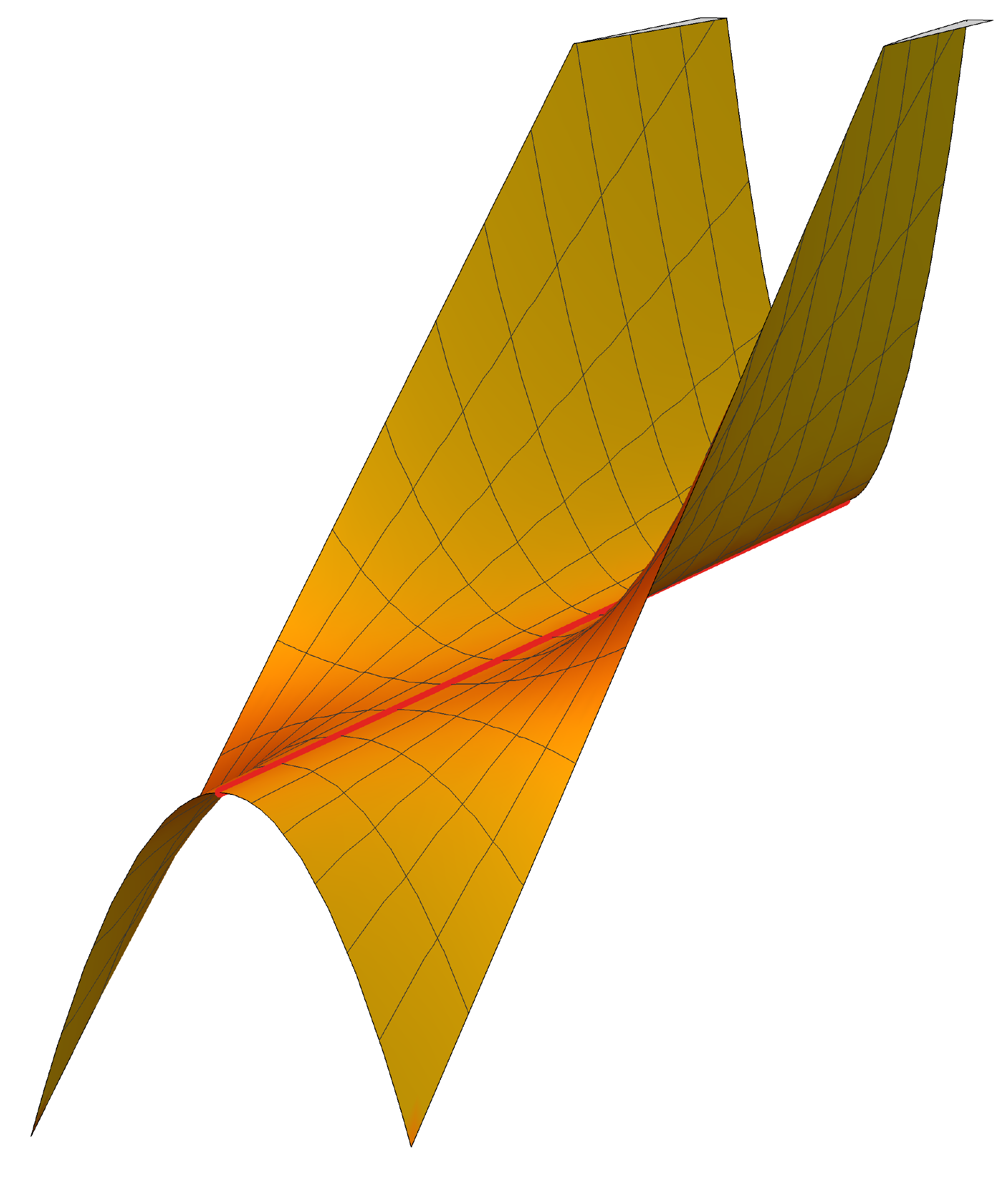}
}
\caption{A ruled canyon which may not appear as part of a non-$1$-parameter line fibration}
\label{fig:canyon}
\end{figure}
 
\begin{proof}
Although relatively clear geometrically, we give an analytic argument for the contrapositive.  We write $\gamma$ for $\gamma_{P,E}$.  Suppose that there exist $t_1 < 0 < t_2$, with $\gamma(t_1) = \gamma(t_2)$ not equal to $u \coloneqq \gamma(0)$.  Let $q_i$ be the intersection of $u^\perp$ with any fiber $\ell_{t_i}$ contained in plane $P_{t_i}$.  Let $\ell$ be any line contained in $P_0$ and label $\ell \cap u^\perp$ as the origin.  By definition of $E$, the map $B$ (see Section \ref{sec:skew}) is well-defined on $E$.  Then $\det( q_1  \ \ B(q_1) )$ has a different sign than $\det( q_2 \ \ B(q_2) )$, because $q_1$ and $q_2$ point in opposite directions with respect to $B(q_1) = B(q_2)$.  Incorporating this sign difference into the remark following Theorem \ref{thm:nononzero}, every path from $q_1$ to $q_2$ within $E$ must intersect $S_u$.

Because $S_u$ does not intersect $P_{t_i}$, we may consider paths from $q_1$ to $q_2$ which first travel some distance along the line $P_{t_1} \cap u^\perp$, then straight to the line $P_{t_2} \cap u^\perp$, then to $q_2$ within $P_{t_2} \cap u^\perp$.  We see that every line in $u^\perp$, through every point of $m \coloneqq P_{t_1} \cap u^\perp$, except for $m$ itself, must intersect $S_u$.  Then the fiber through every point of $m$ must lie in the plane $P_{t_1}$, hence have must have direction $\gamma(t_1)$, and so the fibration is $1$-parameter.
\end{proof}

\subsection{Connectivity of $S_u$ for non-$1$-parameter fibrations}

The connectivity of $S_u$ is fairly technical.  We begin with some elementary statements about convex sets.

\begin{definition}
We say that a line $m \in \R P^1$ \emph{supports} a closed convex set $S \subset \R^2$ if there exists some translate $m'$ of $m$ for which $S$ lies on one side of $m'$.  We say that $m$ \emph{strictly supports} $S$ if additionally $m' \cap \partial S$ is a nonempty bounded set (a point or closed interval).
\end{definition}

\begin{examples}
\begin{enumerate}
 \item If $S$ is bounded, every element $m \in \R P^1$ strictly supports $S$.
 \item If $S$ is the closed convex component ``inside'' one branch of a hyperbola, then the asymptotic directions support $S$, but not strictly.
 \item If $S$ is a ray, then every $m \in \RP^1$, except for the line itself, strictly supports $S$.
 \item If $S$ is a line, then no lines strictly support $S$, but $m = S \in \RP^1$ non-strictly supports $S$.
 \end{enumerate}
 \end{examples}

\begin{lem}
\label{lem:support}
For any closed convex set $S$, there is a partition of $\RP^1$ into three (possibly empty) subsets: the open, connected subset $\mathscr{S}$ of lines which strictly support $S$, the finite set $\mathscr{N}$, contained in $\partial \mathscr{S}$ (if it exists), consisting of lines which non-strictly support $S$, and the connected subset $\mathscr{I}$ of lines for which every translate intersects $S$.  Moreover, $\mathscr{S}$ is nonempty provided that $S$ is not a line.
\end{lem}

\begin{proof}  
If $S$ is bounded, then $\mathscr{S} = \RP^1$, and if $S$ is unbounded with empty interior (a line or ray), then the statement is easy to verify.  It remains to consider the case in which $S$ is unbounded with nonempty interior.  Let $\alpha(t)$ be a parametrization of the boundary of $S$.  The supporting lines may be viewed in relation to the left- and right-hand derivatives of $\alpha$.  In particular, the set of supporting lines $\mathscr{S} \cup \mathscr{N}$ is the interval of $\RP^1$ bounded by the unoriented line $R$ corresponding to the limit, as $t \to -\infty$, of the right-handed derivative of $\alpha$, and by the unoriented line $L$, corresponding to the limit, as $t \to \infty$, of the left-handed derivative of $\alpha$.  The remaining properties are straightforward to verify.
\end{proof}

Using the language above, we now give some intuition for the connectedness of $S_u$.  Suppose that $S$ is a connected component of $S_u$ and $m$ is a strict supporting line.  Then the plane $P$ spanned by $m$ and $u$ is disconnected by fibers, and then there is a (locally) well-defined, continuous parallel plane pushoff map $\gamma_{P,E}$.  When $S$ is bounded, this allows us to push off from $S$ in any direction, so that $u$ is an interior point of the image of the restriction of $V$ to a neighborhood of $S$ in $E$.  Thus continuity at infinity holds for $u$, $S_u = S$, and $-u \notin U$.

But to apply this logic, we must first show that $S$ is isolated; otherwise it is (a priori) possible that $S_u$ is dense around $S$, and $\gamma_{P,E}$ is the constant function $u$.  Assume this isolation, and suppose that $S_1$ and $S_2$ are two unbounded components.  If there exists a line $m$ which strictly supports $S_1$ and $S_2$, then we can push off as above.  Since the parallel plane pushoff map is monotone (Lemma \ref{lem:nocanyon}), the image must include the entire great circle $P \cap S^2$.  But at any intermediate direction $v$ on the great circle, $S_v$ is bounded, so continuity at infinity holds for $v$, hence $-v$ cannot be in the image of the Gauss map, a contradiction.

This section is devoted to making this idea rigorous, beginning with the isolation of any connected component of $S_u$.

\begin{lem}
\label{lem:isolated}
If $S$ is a connected component of $S_u$ in a non-$1$-parameter line fibration, then $S$ is isolated.  That is, for the neighborhood $E$ of $S$ (see Corollary \ref{cor:localppp}), $E \cap S_u = S$.
\end{lem}

\begin{proof}
Suppose for contradiction that $S_1$ and $S_2$ are two connected components of $S_u$ which both define the same neighborhood $E \subset u^\perp$.  In the language developed above, the choice of $E$ ensures that $\gamma_{P,E}$ is defined and continuous for any plane $P$ such that $P \cap u^\perp$ strictly supports either $S_1$ or $S_2$.  Let $\mathscr{L} \subset \RP^1$ be the set of lines such that some translate separates $S_1$ from $S_2$ and does not intersect $S_u$.  By Lemma \ref{lem:discbyline}, $\mathscr{L}$ is nonempty. Moreover, for each $i = 1,2$, $\mathscr{L} \subset \mathscr{S}_i \cup \mathscr{N}_i$; the subscript $i$ indicates that a set is defined with respect to the closed convex subset $S_i$.

If there exists $m \in \mathscr{S}_i \cap \mathscr{L}$ for either $i$, then $m$ is a strict support line for $S_i$.  Thus for $P = \operatorname{Span}\left\{u,m\right\}$, $\gamma_{P,E}$ is well-defined and continuous.  Moreover, since $m \in \mathscr{L}$, $\gamma_{P,E}$ takes at least two distinct values ($u$ and some other value on the translate of $P$ containing the copy of $m$ which separates $S_1$ from $S_2$), contradicting the monotonicity statement of Lemma \ref{lem:nocanyon}.

The remaining possibility is that $\mathscr{L} \subset \mathscr{N}_i$ for both $i$, so $\mathscr{L}$ is either one point or two points.

\emph{Case 1}. Suppose that $\mathscr{L}$ consists of a single point.
 Then $m$ non-strictly supports both $S_1$ and $S_2$ (so that each set looks asymptotically like some translate of $m$), and there exists some translate $\mathfrak{m}$ of $m$ which separates $S_1$ and $S_2$ and does not intersect $S_u$.

Let $p \in \conv(S_1 \cup S_2)$.  Then every line through $p$, except the line parallel to $\mathfrak m$, intersects $S_u$, since $\mathscr{L} = \left\{ m \right\}$.  Thus $p$ must have fiber direction whose projection onto $u^\perp$ is parallel to $m$ (or equal to zero); since otherwise the fiber through $p$ would intersect $V^{-1}(u)$.

Suppose that $\mathfrak{m} \subset \conv(S_1 \cup S_2)$; this occurs, for example, if $S_1$ and $S_2$ diverge in ``opposite directions'' (with respect to $m$).  Then every point of $\mathfrak m$ has fiber contained in the affine copy of $P$ containing $\mathfrak m$, and the fibration is $1$-parameter.  If otherwise $S_1$ and $S_2$ diverge in the same direction (with respect to $m$), then $\conv(S_1 \cup S_2)$ contains a union of rays with direction $m$, all of whose fibers have direction which projects parallel to $m$.  Thus there is a well-defined and continuous parallel pushoff map with respect to $P$ (even though $P$ is not necessarily disconnected by fibers), and $\gamma_{P,E}$ takes at least two distinct values ($u$ and some other value on the translate of $P$ containing $\mathfrak m$), contradicting the monotonicity statement of Lemma \ref{lem:nocanyon}.

\emph{Case 2}.  Suppose that $\mathscr{L}$ consists of two points $m_1$ and $m_2$, then $\mathscr{L} = \mathscr{N}_1 = \mathscr{N}_2$.  Then every line through every point of $\conv(S_1 \cup S_2)$ has fiber direction whose projection onto $u^\perp$ is parallel to either $m_1$ or $m_2$ (or equal to zero).  Consider the copy $\mathfrak{m}_1$ of $m_1$ which separates $S_1$ and $S_2$.

We claim that every point of $\mathfrak{m}_1 \cap \conv(S_1 \cup S_2)$ has fiber contained in the affine copy of $P$ containing $\mathfrak{m}_1$.  Indeed, the projection of any fiber to $ u^\perp$ cannot be zero, since $\mathfrak{m}_1$ is disjoint from $S_{\pm u}$.  Moreover, the intersection is either a line or ray, either of which is connected, and so the projection cannot ``switch'' from parallel-to-$m_1$ to parallel-to-$m_2$, by continuity.  Thus we are in the situation of Case 1, and the proof is complete.
\end{proof}

\begin{lem}
\label{lem:strictdisc}
If $S_1$ and $S_2$ are any two distinct connected components of $S_u$ in some non-$1$-parameter fibration, then there exists a line separating $S_1$ and $S_2$ which does not intersect $S_u$ and strictly supports $S_1$ and $S_2$.
\end{lem}

\begin{proof}
The proof is similar to the proof of Lemma \ref{lem:isolated}.  First we choose a separating line $m$ which may fail to strictly support either $S_1$ or $S_2$.  Then by Lemma \ref{lem:support}, either $m$ can be perturbed to an $m'$ which strictly supports both, or $m$ can be perturbed to $m'$ which strictly supports $S_1$ and such that the translate of $m'$ which intersects $\partial S_1$ intersects the interior of $S_2$.  In this latter case, we push off from $S_1$ with respect to the plane $P$ spanned by $u$ and $m'$.  By Lemma \ref{lem:isolated}, $S_1$ is isolated, so $\gamma_{P,E}(\varepsilon) \neq u$, but this means that the fiber in $P_\varepsilon$ will intersect $S_2 \times \R$.
\end{proof}

\begin{lem}
\label{lem:continffinite}
Suppose that for some $u \in U$, $S_u$ is compact.  Then continuity at infinity holds for the direction $u$, $S_u = S$, and $-u \notin U$.
\end{lem}

\begin{proof}
Because $S_u$ is bounded, every line $m \in \RP^1$ strictly supports $S_u$, and the parallel plane pushoff $\gamma_{P,E}$ is well-defined for every $P$ containing $u$.  By Lemma \ref{lem:isolated}, $\gamma_{P,E}$ is nonconstant as we push off from $S_u$ with respect to $P$, and Lemma \ref{lem:nocanyon} ensures that $u$ is an interior point of $V(E)$. The idea of Lemma \ref{lem:continfskew} goes through unchanged, so continuity at infinity holds for $u$.  In particular, the fibers emanating from $E$ swallow any sufficiently far double-cone as depicted in Figure \ref{fig:contatinf}, hence there may be no fibers outside $E$ with direction $\pm u$, since any such fiber would intersect such a cone.
\end{proof}

\begin{lem}
\label{lem:connected}
For any $u \in U$, the set $S_u$ is either connected or is a union of parallel lines.
\end{lem}

\begin{proof}
If $S_u$ does not contain a parallel line, then we are not in the $1$-parameter situation, and Lemma \ref{lem:strictdisc} applies.  In particular, we assume that there exist connected components $S_1$ and $S_2$ of $S_u$, with corresponding $E_1$ and $E_2$, and there exists a line $m$ which disconnects $S_1$ and $S_2$, strictly supports $S_1$ and $S_2$, and does not intersect $S_u$.

Suppose that $m$ is not parallel to a line in $\partial E$.  Then for every $t$, $P_t \cap E$ is not empty and there exists a fiber emanating from it which is contained in $P_t$.  Moreover, if $t$ is such that $P_t \cap S_1 = \emptyset$, then the fiber direction is not equal to $u$.  However, there exists a $t$ such that $P_t \cap u^\perp$ is a strict supporting line for $S_2$, hence $P_t$ contains a fiber with direction $u$, a contradiction.

Suppose otherwise that $m$ may be chosen as a line in $\partial E$, so that $m$ is an actual fiber.  Let $P = \operatorname{Span}\left\{m,u\right\}$ and consider the parallel plane pushoff $\gamma_{P,E_1}$, starting at $S_1$ and pushing towards $m$.  The image traces out an arc of the great circle $P \cap S^2$ which starts at $u$ and continues rotating monotonically (by Lemma \ref{lem:nocanyon}) until arriving at the fiber $m$, after which the rotation must continue (by applying Lemma \ref{lem:nocanyon} to the direction of $m$) until eventually reaching $u$ again as the pushoff planes hit $S_2$.  Hence the great circle $c = P \subset S^2$ is inside $U$.  But for a point $v \in c$, such that $v \neq \pm u$ and $v$ not parallel to $m$, $V^{-1}(v) \cap u^\perp$ must contain a bounded component, since $P$ is disconnected by fibers, hence must be bounded, by Lemma \ref{lem:continffinite}.  Again by Lemma \ref{lem:continffinite}, $-v \notin U$, a contradiction.
\end{proof}

We are finally ready to prove Theorem \ref{thm:main}.

\begin{proof}[Proof of Theorem \ref{thm:main}]

The first item follows from Lemmas \ref{lem:convex1} and \ref{lem:connected}.  The second item is Lemma \ref{lem:continffinite}.  For the third item, the parallel plane pushoff is well-defined and continuous with respect to every plane $P \in \Gr_2(3)$, and so the proof Lemma \ref{lem:convex} goes through verbatim.  For the fourth item, if some $S_u$ is not a line, then there exists a strict supporting line, hence a plane which is disconnected by fibers, hence a well-defined parallel plane pushoff map and some $v$ in the image with compact $S_v$.

\end{proof}

\subsection{Additional results when there exist multiple noncompact base spaces}
\label{sec:additional}

The most mysterious of the line fibrations are those containing noncompact $S_u$ for multiple $u \in U$.  For example, consider a fibration which contains a fibered half-plane, such that the fibers have direction $u$ and $S_u$ is a ray with direction $m$.  Then $P = \Span\left\{u, m\right\}$ is not disconnected by fibers, and so it is possible that the parallel plane pushoff map is not defined.  Even when defined (for example, if the fibration contains a family of parallel fibered half-planes) we do not know if the map is continuous.  These difficulties have prevented us from classifying fibrations with multiple unbounded bases.  However, we offer a handful of results which we hope could contribute to proofs of Conjectures \ref{conj:one} and \ref{conj:two}.

\begin{lem}
\label{lem:antipodal}
Suppose that a fibration is not $1$-parameter and that exists an antipodal pair $\pm u \in U$.  Then $S_u$ and $S_{-u}$ must be unbounded and have the same set of strict supporting lines.  Moreover, given a representation of $S_u$ as the intersection of closed half-planes $H_\iota$, the set $S_{-u}$ must be contained in the intersection of $\bar{H_\iota^c}$, i.e.\ the closure of the complementary half-planes; and vice versa.
\end{lem}

\begin{proof}
The unboundedness follows from Lemma \ref{lem:continffinite}.  The proof that the sets of strict supporting lines match is identical to the proof of Lemma \ref{lem:isolated}: if there is a strict support line for $S_u$ which does not strictly support $S_{-u}$, then fibers in the push-off from $S_u$ would intersect $V^{-1}(-u)$.  Moreover, $S_u$ may be written as the intersection of half-planes corresponding to the strict support directions.  If $S_{-u}$ is not contained in the intersection of the closure of the complements, then it intersects a support line of $S_u$, and again a pushoff fiber intersects $V^{-1}(-u)$.
\end{proof}

\begin{lem}
Suppose that there exist $\pm u \in U$ such that $S_u$ and $S_{-u}$ are each a ray.  Then $S_u$ and $S_{-u}$ are contained in a single line.
\end{lem}

\begin{proof}
Suppose otherwise, that there exists a fibered half-plane with direction $u$ and a fibered half-plane with direction $-u$.  Let $H_\iota$ be the supporting half planes whose intersection gives the ray $S_u$.  Because $S_{-u}$ must lie in the intersection $\bar{H}_{\iota}^c$, $S_{-u}$ must be in the line determined by the ray $S_u$.
\end{proof}

\begin{lem}
Suppose that there exist $u_0, u_1 \in U$ not antipodal with both $S_{u_i}$ noncompact.  Let $m_i \in \RP^1$ be the element obtained by projecting $u_{1-i}$ to $u^\perp_i$.  Then $m_i$ either strictly supports or nonstrictly supports $S_{u_i}$.
\end{lem}

\begin{proof}
If every translate of $m_i$ intersects $S_{u_i}$, then any fiber with direction $u_{1-i}$ intersects $V^{-1}(u_i)$.
\end{proof}

\begin{lem}
\label{lem:nonstrictantipodal}
Suppose that there exist $\pm u_0, u_1 \in U$ with $S_{u_1}$ unbounded.  Let $m \in \RP^1$ be the element obtained by projecting $u_{1}$ to $u^\perp_0$.  Then $m_i$ must nonstrictly support $S_{\pm u_0}$.
\end{lem}

\begin{proof}
If $m_i$ strictly supports $S_{\pm u_0}$, then $P = \Span \left\{ u_0, u_1 \right\}$ is disconnected by fibers, and the closed great semicircle through $\pm u_0$ and $u_1$ is a subset of $U$.  Moreover, every element besides $\pm u_0$ must have bounded base (otherwise $P$ would not be disconnected by fibers).
\end{proof}

It seems plausible that the pushoff map could be utilized further to Conjecture \ref{conj:two}, that multiple pairs of antipodal points cannot exist outside of a $1$-parameter fibration.  For example, if there exist multiple pairs $\pm u_0, \pm u_1 \in U$, then combining the results of Lemmas \ref{lem:antipodal} and \ref{lem:nonstrictantipodal} gives fairly severe restrictions on possible base spaces $S_{\pm u_i}$.

\section{Fibrations and contact structures}
\label{sec:contact}
Here we discuss the relationship between line fibrations and contact structures.  We begin with the proof of Theorem \ref{thm:nononzero}.

\begin{proof}[Proof of Theorem \ref{thm:nononzero}]

Given a smooth fibration of $\R^3$ by oriented lines, let $p \in \R^3$ be a point with fiber direction $u$, and consider $B$ defined in a neighborhood $E$ of $p$ in $u^\perp$.  Let $u^\perp_t$ be the translate of $u^\perp$ by distance $t$ for $t \neq 0$.  The map $f : E \to u^\perp_t : q \mapsto q + tB(q)$, which sends a point $q \in E$ to the intersection with $u^\perp_t$ of the fiber through $q$, is a diffeomorphism onto its image, as its inverse could be defined in exactly the same manner.  Thus $df_p = \operatorname{Id} + tdB_p$ is an invertible linear map for all $t \neq 0$, hence $dB_p$ has no nonzero real eigenvalues.

Regarding the final assertion of Theorem \ref{thm:nononzero}: if $V$ is not nondegenerate at $p$, then the dimension of its kernel is at least $2$, so the third eigenvalue must be real, and it cannot be nonzero.
\end{proof}

\begin{proof}[Proof of Theorem \ref{thm:contact}]  We adopt the notation of \cite{Harrison2}, in which the contact assertion was shown for nondegenerate fibrations.  Let $p \in \R^3$, let $\alpha$ be the $1$-form dual to $V$, and let $B : E \to \R^2$ be defined as above.  Let $\R^2 \subset \R^3$ be the linear subspace parallel to $u^\perp$.  For $B(p) \in \R^2$ and $h \in T_pE$, we use bold letters ${\bf B}(p)$ and ${\bf h}$ to represent the associated (by inclusion) vectors in $\R^3$.  This allows us to write, for $q \in E$, 
\[
{\bf B}(q) = \frac{1}{\langle V(q), V(p) \rangle}V(q) - V(p).
\]
We compute
\[
d{\bf B}_q h = \frac{1}{\langle V(q), V(p) \rangle}dV_q{\bf h} - \frac{\langle dV_q {\bf h}, V(p) \rangle}{\langle V(q), V(p) \rangle^2}V(q),
\]
where $h \in T_q E$.  Evaluating at $q = p$ yields
\[
d{\bf B}_p h = dV_p {\bf h},
\]
as $V$ is a unit vector field, and $dV_p {\bf h}$ is orthogonal to $V(p)$.

Define the quadratic form $Q(h) = \det\left( \ h \ \ dB_p h \ \right)$.  By the computation above, and by the fact that ${\bf h}$ and $d{\bf B}_p h$ are orthogonal to the oriented unit vector $V(p)$, we may write:
\[
Q(h) = \det\big( \ h \ \ dB_p h \ \big) = \det\big( \ {\bf h} \ \ dV_p{\bf h} \ \ V(p) \ \big) = \langle {\bf h} \times dV_p{\bf h}, V(p) \rangle,
\]
so that
\begin{align}
\label{eqn:trace}
\operatorname{trace}(Q) = \langle \operatorname{curl}(V(p)), V(p) \rangle = \ast (\alpha(p) \wedge d\alpha(p)).
\end{align}

In case $dB_p$ has rank $2$, $V$ is nondegenerate and the assertion was shown in \cite{Harrison2}.  In case $dB_p$ has rank $1$, it follows from Theorem \ref{thm:nononzero} that its only real eigenvalues are zero, hence it has zero trace and zero determinant, but it is not the zero operator.  Thus in any local coordinates it has one of the following forms for $a$ and $b$ nonzero:

\[
\left( \begin{array}{cc} a & -b \\ \frac{a^2}{b} & -a \end{array} \right), \ \ \left( \begin{array}{cc} 0 & b \\ 0 & 0 \end{array}\right), \ \ \left( \begin{array}{cc} 0 & 0 \\ b & 0 \end{array} \right).
\]

Computing $\operatorname{trace}(Q) = \operatorname{trace}(h \mapsto \det\big( \ h \ \ dB_p h \ \big))$ in these three cases yields, respectively, $\frac{a^2+b^2}{b}$, $-b$, and $b$; any of which is nonzero, hence the form is contact by (\ref{eqn:trace}).

Conversely, if $V$ is not semidegenerate, then $\nabla V_p$ is $0$ at some point, as is $d\alpha(p)$, and the plane distribution is not contact.

It remains to show that the contact structure is tight.  It was shown in \cite{Harrison2} that if a fibration corresponds to a contact structure, and if there exists any fiber which admits no parallel fiber, then the contact structure is tight.

Note moreover that if a fibration induces a contact structure, then $V$ is not semidegenerate at any point, so no $S_u$ may have an interior point.  Therefore each $S_u$ is either a point, a line segment, a ray, or a union of one or more lines.  If any $S_u$ is a point, the above logic applies and the contact structure is tight.  If any $S_u$ is a line, then the fibration is $1$-parameter and embeds in the standard contact structure.

We show that there are no other possibilities.  This was recently shown by Becker and Geiges \cite{BeckerGeiges} using a different argument.  Assume for contradiction that every $S_u$ is either a line segment or a ray.  Either way, every point $p \in \R^3$ is a critical point of $V\big|_P$, where $P$ is any plane transverse to $V(p)$.  Then applying Sard's theorem to the restriction yields that $V(P)$, and hence $U$, is measure zero.  If some $S_u$ is a line segment, then we may push off in any direction to see that $u$ is an interior point of $U$, a contradiction.  On the other hand, if every $S_u$ is a ray, then we may choose $u'$ and push off of the boundary point of $S_{u'}$, in the direction of any strict support plane $P$, to obtain fibers in $P_t$ which disconnect $P_t$, contradicting that every $S_u$ is a ray.
\end{proof}

\section{Examples}
\label{sec:examples}

We will generate examples of nonskew fibrations with the use of the following result.  Given $E$ an open subset of $P_0 \simeq \R^2$ (one may assume that $E$ is an open half-plane, a strip between two parallel lines, or $P_0$ itself), consider a map $B : E \to \R^2$ and the collection of lines $\left\{ (p,0) + t(B(p),1) \ | \ p \in E \right\}$.

\begin{lem}
\label{lem:harrison}
Let $E \subset \R^2$ and $B : E \to \R^2$ be continuous.  Then the collection of lines $\left\{ (p,0) + t(B(p),1) \ | \ p \in E \right\}$ fiber $\R^3 - (P_0 - E)$ if and only if
\begin{itemize}
\item for any distinct points $p$ and $q$ in $E$, $p-q$ is not a multiple of $B(p) - B(q)$, and 
\item if $p_n$ is a sequence of points in $E$ with no accumulation points in $E$, then $|p_n + tB(p_n)| \to \infty$ for all fixed $t \neq 0$.
\end{itemize}
\end{lem}

Observe that when $E = \R^2$, $B$ describes a fibration of $\R^3$ by lines.

The analogous lemma in the case of skew fibrations was shown in \cite{Harrison}, and the proof is nearly identical, so we offer only the main idea.  Consider the map $f_t : E \to \R^2 : p \mapsto p + tB(p)$, and assume that the two items hold.  The first bullet point is equivalent to injectivity of the map $f_t$.  Then $f_t$ is continuous and injective and so it is a homeomorphism onto its image; in particular the image is open.  Finally, it is straightforward to check from the second bullet point that the image of $f_t$ is closed for each $t \neq 0$.  So for $t \neq 0$, each $f_t$ is surjective, and each point in the plane $P_t$ is hit by a line from the collection.

\begin{example} For $E = \R^2$ and $H(p) = \pm i p$, the resulting fibration is the standard Hopf fibration.
\end{example}

We compare the Hopf map to the first bullet point of Lemma \ref{lem:harrison}.  The fibers through $p$ and $q$ do not intersect provided that $p - q$ is not a multiple of $B(p) - B(q)$.  Note that in the Hopf case, $H(p) - H(q) = i(p-q)$ is always perpendicular to $p - q$.  Thus we can generate line fibrations of the form $B = H \circ f$ using maps $f : \R^2 \to \R^2$ satisfying the properties that 
\begin{enumerate}
\item for all $p \neq q$ with $f(p) \neq f(q)$, $p - q$ is not orthogonal to $f(p) - f(q)$, and
\item outside of some bounded set, the behavior of $f$ matches that of the identity function.
\end{enumerate}
Indeed, each item guarantees the respective bullet point for $H \circ f$.  Although the second item is somewhat nonspecific, it is clearly satisfied in the examples given below.

\begin{example}  Let $f: \R^2 \to \R^2$ be the polar map $(r,\theta) \mapsto (\max\left\{0,r-1\right\}, \theta)$, that is, $f$ maps the closed unit disk to the origin and maps every larger concentric circle to a circle of one less radius.  The composition $H \circ f$ generates the ordinary Hopf fibration, except that the vertical line in the center is now a vertical cylinder of radius $1$.  We may show the first bullet point above, that there exists no vector $p_1-p_2$ for which $f(p_1) - f(p_2) \neq 0$ is perpendicular to $p_1 - p_2$.  Write $p_i = (r_i \cos \theta_i, r_i \sin \theta_i)$.  We will assume that $r_i > 1$, so $f(p_i) = ((r_i-1)\cos \theta_i, (r_i - 1)\sin \theta_i)$.  Then we have
\[
\langle p_1 - p_2, f(p_1) - f(p_2) \rangle = r_1(r_1 - 1) + r_2(r_2-1) + r_1(1 - r_2)\cos(\theta_1 - \theta_2) + r_2(1-r_1)\cos(\theta_1 - \theta_2),
\]
which one can check is positive when $p_1 \neq p_2$.  Indeed, one can first explicitly check positivity when $r_1 = r_2$ and $\theta_1 \neq \theta_2$.  Next, for $r_1 \neq r_2$ and $\cos(\theta_1 - \theta_2) \leq 0$, each individual nonzero term is positive.  Finally, if $r_1 \neq r_2$ and $\cos(\theta_1 - \theta_2) > 0$, then the expression is greater than or equal to
\[
[r_1(r_1 - 1) + r_2(r_2-1) + r_1(1 - r_2) + r_2(1-r_1)]\cos(\theta_1 - \theta_2) = (r_1-r_2)^2\cos(\theta_1-\theta_2) > 0.
\]
\end{example}

\begin{example}
The above example can be made $C^\infty$ by redefining the function on the set $r > 1$ as $(r,\theta) \mapsto (e^{-\frac{1}{(r-1)^2}}, \theta)$.
\end{example}

\begin{example}
Let $C$ be a compact, strictly convex set with nonempty interior and $C^1$ boundary.  Define $f$ identically zero on $C$.  For $p \in \R^2 - C$, write $p = c + tn(c)$, where $c \in \partial C$, $n(c)$ is the outward-pointing unit normal vector at $c$, and $t > 0$; then let $f(p) = tn(c)$.  As in the previous examples, the behavior of $f$ near infinity resembles that of the identity, so we only must check the first condition on $f$.

First let $p \notin C$ and $q \in C$.  We write $p = c + tn(c)$ for some $c \in C$.  We must check that $p - q$ is not orthogonal to $f(p) - f(q) = tn(c)$.  Since $p$ lies in the half-plane $H$ supporting $C$ at $c$, and the boundary of this half-plane is orthogonal to $n(c)$, the line through $p$ orthogonal to $n(c)$ is fully contained in $H$, hence may not contain $q$.

Now for $i = 1, 2$ consider $p_i = c_i + t_in(c_i)$ distinct points not in $C$, and suppose that the fibers through $p_1$ and $p_2$, of the fibration defined by $B = H \circ f$, intersect.  Then $p_1 - p_2 = \lambda i (t_1n(c_1) + t_2n(c_2))$ for some $\lambda \neq 0$, hence $p_1 - p_2 = \lambda(t_1 \gamma'(s_1) + t_2\gamma'(s_2))$, where $\gamma : I \to \partial C$ is some unit-speed parametrization of the boundary with $\gamma(s_i) = c_i$, and chosen with the orientation such that $\lambda$ is negative.  Then $p_1 - \lambda t_1 \gamma'(s_1) = p_2 - \lambda t_2 \gamma'(s_2)$.  Let $r$ be this point on both sides of the equation.  Geometrically, $r$ is the third point of the two right triangles $\Delta(c_i p_i r)$, as decpited in Figure \ref{fig:notintersect}.

\begin{figure}[h!t]
\centerline{
\includegraphics[width=3.5in]{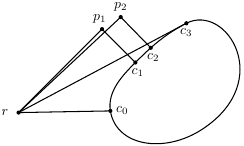}
}
\caption{Showing two fibers of $B$ will not intersect}
\label{fig:notintersect}
\end{figure}

We have

\[
\frac{\mbox{length}(c_i,p_i)}{\mbox{length}(p_i,r)} = \frac{t_i}{-\lambda t_i} = - \frac{1}{\lambda},
\]
for each $i$, and we would like to obtain a contradiction.

As depicted in Figure \ref{fig:notintersect}, let $c_0$ be the point of $\partial C$ at which the normal vector hits $r$, and let $c_3$ be the point of $\partial C$ so that the negatively-oriented tangent ray hits $r$.  Then $c_1$ and $c_2$ must lie on the arc of $\partial C$ traversed positively by $\gamma$ from $c_0$ to $c_3$.  Let $s_i = \gamma^{-1}(c_i)$.

For $s \in (s_0,s_3)$, let $p_s$ be the point of $\R^2$ such that $\Delta(\gamma(s) p_s r)$ is a right triangle.  Define $\beta : (s_0,s_3) \to \R$ as the angle $\angle p_s r \gamma(s)$.  Then $\beta$ increases strictly monotonically from $0$ to $\frac{\pi}{2}$, hence so does $\tan \beta$.  Thus it is impossible that $\tan \beta = -\frac{1}{\lambda}$ at two points.
\end{example}

We suspect that the above example can be modified to allow for $C$ to be a line segment or ray.

\begin{example}
\label{ex:halfhalf}
We construct a fibration which is half of a Hopf fibration and half of a $1$-parameter fibration.  Consider the image of the $p_2$-axis for the standard Hopf fibration $H(p) = ip$.  The resulting fibers generate a helicoid which rotate (from the perspective of the origin), $\frac{\pi}{2}$ in the positive $p_2$ direction and $\frac{\pi}{2}$ in the negative $p_2$ direction.  This helicoid disconnects $\R^3$, so that on either side there is a ruled open solid.  We may replace the entire "right" half of the fibration (for $p_1$ positive) with a copy of the same helicoid at every $p_1$.  That is, the fibration is the Hopf fibration for $p_1$ negative, and for $p_1$ positive it is a fibration by a one-parameter family of half-planes.  Algebraically, $f(p_1,p_2)$ is the identity for $p_1 \leq 0$ and maps to $(0,p_2)$ for $p_1 \geq 0$.
\end{example}

\begin{example}
Similar to the above example, instead of replacing the right side, we may push the two halves apart some finite distance and insert a fat helicoid.  The generating map is $f(p_1,p_2) = (p_1 - 1, p_2)$ for $p_1 \geq 1$, $(p_1 + 1, p_2)$ for $p_1 \leq -1$, and $(0,p_2)$ in between.
\end{example}

\subsection{Exotic example}

By construction, all of the fibrations above have convex Gauss image.  The following example is not $1$-parameter, but nevertheless contains a pair of antipodal points in the Gauss image.
\label{sec:exotic}

Consider the horizontal plane $P \subset \R^3$ with rectangular coordinates $(x,y)$.  Cover the half-plane $x \geq \frac{\pi}{2}$ with lines pointing in the positive $y$ direction, and cover the half-plane $x \leq - \frac{\pi}{2}$ with lines pointing in the negative $y$ direction.  On the strip $E \coloneqq \left\{ (x,y) \in P \ \big| \ -\frac{\pi}{2} < x < \frac{\pi}{2} \right\}$, define $B(x,y) = (-y,\tan x)$, and consider the collection of parametrized oriented lines $(x,y,0) + t(B(x,y),1)$ emanating from points of $E$.  We aim to show that this defines a non-skew fibration with the image of the Gauss map an open hemisphere together with two antipodal points.

We first check Lemma \ref{lem:harrison}.  For $p \neq q$ in $E$,

\[
\det \left( \begin{array}{cc} p_1 - q_1 & -(p_2 - q_2) \\ p_2 - q_2 & \tan p_1 - \tan q_1 \end{array} \right) = (p_1 - q_1)(\tan p_1 - \tan q_1) + (p_2 - q_2)^2.
\]
By monotonicity of $\tan$, both summands are non-negative, and both equal $0$ if and only if $p=q$.  Therefore the lines emanating from $E$ do not intersect.

Now let us check that the collection of lines covers all of $\R^3$.  We must show that for fixed $t$, the map $(p_1,p_2) \mapsto (p_1,p_2,0) + t(B(p_1,p_2),1)$ is onto the plane $z = t$.  That is, the map $(p_1,p_2) \mapsto (p_1-tp_2,p_2+t\tan p_1)$ is onto for every fixed $t \neq 0$.  For this we must solve:
\begin{align*}
p_1 - tp_2 & = x \\
p_2 + t\tan p_1 & = y.
\end{align*}
By adding the first equation to a multiple of the second, we obtain
\begin{align}
\label{eqn:solvep1}
p_1 + t^2\tan p_1 = x + ty.
\end{align}
The map $(-\frac{\pi}{2},\frac{\pi}{2}) \to \R : p_1 \mapsto p_1 + t^2\tan p_1$ is a homeomorphism, thus we may solve (\ref{eqn:solvep1}) for $p_1$.  Then either equation gives a value for $p_2$.

Finally, let us check that the covering of lines is continuous at the points of $P - E$.  For this we must check that if $(x_n,y_n,z_n)$ is any convergent sequence of points with $z_n \to 0$ and $x_n \to C$ with $|C| \geq \pi$, then the directions $u_n$ of the fibers through $(x_n,y_n,z_n)$ converge to $\pm u \coloneqq \pm (0,1,0)$, where the sign depends on the sign of $C$.  Assume $x_n \to C \geq \pi$ and suppose that $z_n \neq 0$.  For each $n$ we may solve (\ref{eqn:solvep1}) to find a sequence $p_{n_1}$ such that $p_{n_1} + z_n^2\tan p_{n_1} = x_n + z_ny_n$.  By assumption, the limit of the right side is $C \geq \frac{\pi}{2}$.  Therefore, the limit of the left side exists and equals $C$.

Consider any convergent subsequence of $p_1^n$, and suppose it limits to $D$.  If $D < \frac{\pi}{2}$, then $\tan p_1^n$ converges, but since $z_n \to 0$, we obtain the limit of the left side is $D$, a contradiction.  So any convergent subsequence must converge to $\frac{\pi}{2}$.

We still must show that the sequence of directions does not diverge.  We see from (\ref{eqn:solvep1}) that $z_n^2 \tan p_{n_1}$ converges to something nonzero, so $\tan p_{n_1}$ diverges at the rate at which $z_n^2$ converges.  We also see that $z_np_{n_2}$ converges to zero, so $p_{n_2}$ diverges at the same rate at which $z_n$ converges.

Now

\[
u_n = \frac{(-p_{n_2}, \tan p_{n_1}, 1)}{\sqrt{ 1 + p_{n_2}^2 + \tan^2(p_{n_1})} }.
\]

Using the fact that the first component behaves as $\frac{1}{z_n}$ and the second behaves as $\frac{1}{z_n^2}$, this converges to $(0,1,0)$.

\section{Conclusion}

The main technique for studying nonskew fibrations is to utilize continuity of the parallel plane pushoff maps.  It is not clear whether we have used this technique to its full potential.  In particular, we would like to understand whether the map is well-defined in cases when the plane is not disconnected by fibers; for example, when can we push off of $S_u$ with respect to $P = \Span\left\{ u, m \right\}$ when $m$ is a nonstrict support line of $S_u$?  If the map is well-defined, must it be continuous?

We believe that this and other clever applications of the parallel plane pushoff could yield Conjectures \ref{conj:one} and \ref{conj:two} and complete the classification of nonskew fibrations in four categories: skew-like, $1$-parameter, half-and-half, and exotic.

We offer two other questions proposed to us by Albert Fathi.

What conditions on an arbitrary finite configuration of nonintersecting lines guarantee that the collection can be completed to a fibration?  Given a partition of $\R^2$ into closed convex sets, what conditions guarantee that these sets are the preimages of points for some map $B$ which corresponds to a fibration?

\bibliographystyle{plain}
\bibliography{bib}{}

\end{document}